\documentclass[reqno, 11pt, a4paper]{amsart}

\usepackage{latexsym,ifthen,xspace}
\usepackage{amsmath,amssymb,amsthm}
\usepackage{bbm}
\usepackage{enumerate, calc}
\usepackage[colorlinks=true, pdfstartview=FitV, linkcolor=blue,
  citecolor=blue, urlcolor=blue, pagebackref=false]{hyperref}
\usepackage[text={33pc,605pt},centering]{geometry}    

\newtheorem{theorem}{Theorem}[section]
\newtheorem{lemma}[theorem]{Lemma}

\newtheorem{assumption}[theorem]{Assumption}
\newtheorem{corro}[theorem]{Corollary}

\theoremstyle{definition}

\newtheorem{example}[theorem]{Example}

\theoremstyle{remark}
\newtheorem{remark}[theorem]{Remark}

\numberwithin{equation}{section}

\DeclareMathAlphabet{\mathsl}{OT1}{cmss}{m}{sl}
\SetMathAlphabet{\mathsl}{bold}{OT1}{cmss}{bx}{sl}

\newcommand{\al}{\ensuremath{\alpha}}
\newcommand{\be}{\ensuremath{\beta}}
\newcommand{\de}{\ensuremath{\delta}}
\renewcommand{\th}{\ensuremath{\theta}}
\newcommand{\la}{\ensuremath{\lambda}}
\newcommand{\si}{\ensuremath{\sigma}}
\newcommand{\om}{\ensuremath{\omega}}

\newcommand{\vr}{\ensuremath{\varrho}}
\newcommand{\Si}{\ensuremath{\Sigma}}
\newcommand{\Om}{\ensuremath{\Omega}}
\newcommand{\cC}{\ensuremath{\mathcal C}}
\newcommand{\cF}{\ensuremath{\mathcal F}}
\newcommand{\cL}{\ensuremath{\mathcal L}}
\newcommand{\cO}{\ensuremath{\mathcal O}}
\newcommand{\bbN}{\ensuremath{\mathbb N}} 
\newcommand{\bbR}{\ensuremath{\mathbb R}} 
\newcommand{\bbZ}{\ensuremath{\mathbb Z}} 

\newcommand{\me}{\ensuremath{\mathrm{e}}}
\newcommand{\md}{\ensuremath{\mathrm{d}}}

\DeclareMathOperator{\mean}{\mathbb{E}}
\DeclareMathOperator{\Mean}{\mathrm{E}}
\DeclareMathOperator{\prob}{\mathbb{P}}
\DeclareMathOperator{\Prob}{\mathrm{P}}

\newcommand{\ldef}{\ensuremath{\mathrel{\mathop:}=}}
\newcommand{\rdef}{\ensuremath{=\mathrel{\mathop:}}}

\newcommand{\indicator}{\ensuremath{\mathbbm{1}}}

\begin{document}

\title[Green kernel asymptotics for two-dimensional random walks ]{Green kernel asymptotics for two-dimensional random walks under random conductances}


\author{Sebastian Andres}
\address{University of Manchester}
\curraddr{Department of Mathematics,
Oxford Road, Manchester M13 9PL}
\email{sebastian.andres@manchester.ac.uk}
\thanks{}


\author{Jean-Dominique Deuschel}
\address{Technische Universit\"at Berlin}
\curraddr{Strasse des 17. Juni 136, 10623 Berlin}
\email{deuschel@math.tu-berlin.de}
\thanks{}

\author{Martin Slowik}
\address{Technische Universit\"at Berlin}
\curraddr{
  Strasse des 17. Juni 136, 10623 Berlin
}
\email{slowik@math.tu-berlin.de}
\thanks{}



\date{\today}

\dedicatory{}

\begin{abstract}
  We consider random walks among random conductances on $\bbZ^2$ and establish precise asymptotics for the associated potential kernel and the Green's function of the walk killed upon exiting  balls.  The result is proven for random walks  on i.i.d.\ supercritical percolation clusters among ergodic degenerate conductances satisfying a moment condition.  We also provide a similar result for the time-dynamic random conductance model. As an application we present a scaling limit for the variances in the Ginzburg-Landau $\nabla \phi$-interface model.
\end{abstract} 

\maketitle


\section{Introduction}
We consider the Euclidean lattice $(\bbZ^d, E_d)$ with $d \geq 2$.  The edge set, $E_d$, of this graph is given by the set of all non-oriented nearest neighbour bonds, that is $E_d \ldef \{\{x, y\} : x, y \in \bbZ^d, |x - y| = 1\}$.  We will also write $x \sim y$ if $\{x, y\} \in E_d$.  Consider a family of non-negative weights $\om = \{ \om(e) \in [0,\infty) : e \in E_d\} \in \Om$, where $\Om = [0, \infty)^{E_d}$ is the set of all possible configurations.  We also refer to $\om(e)$ as the \emph{conductance} of the edge $e$.  We call an edge $e \in E_d$ \emph{open} if $\om(e) > 0$ and denote by $\cO(\om)$ the set of open edges.  Let us further define the measure $\mu^{\om}$ on $\bbZ^d$ by $\mu^{\om}(x) \ldef \sum_{y \sim x}\, \om(\{x,y\})$,  and for any $z \in \bbZ^d$ we denote by $\tau_z\!: \Om \to \Om$  the space shift by $z$ defined by
\begin{align*}
  (\tau_{z}\, \om)(\{x,y\})
  \;\ldef\;
  \om(\{x+z,y+z\}),
  \qquad \forall\; \{x,y\} \in E_d.
\end{align*}
We equip $\Om$ with a $\si$-algebra $\cF$.  Further, we will denote by $\prob$ a probability measure on $(\Om, \cF)$, and we write $\mean$ for the expectation with respect to $\prob$.

Throughout the paper we assume that $\mathbb{P}$ is ergodic and that $\prob$-a.s.\ there exists a unique infinite cluster $\cC_{\infty}$ of open edges and $\prob[0\in \cC_\infty]>0$.  For instance, in the case of i.i.d.\ conductances this is fulfilled if $\prob[\om(e) > 0] > p_c$, where $p_c = p_c(d)$ denotes the critical probability for bond percolation on $\bbZ^d$.  Write $\prob_0[ \,\cdot\, ] \ldef \prob[ \cdot \,|\, 0 \in \cC_{\infty} ]$. 

We now introduce the \emph{random conductance model (RCM)}.  Consider a continuous-time Markov chain, $X \equiv \big(X_t : t \geq 0 \big)$, on $\cC_{\infty}(\om)$ with generator $\cL^{\om}$ acting on bounded functions $f\!: \cC_{\infty}(\om) \to \bbR$ as
\begin{align} \label{def:L}
  \big(\cL^{\om} f)(x)
  \;=\; 
  \frac{1}{\mu^\om(x)}\, \sum_{y \sim x} \om(\{x,y\})\, \big(f(y) - f(x)\big).
\end{align}
Then, the Markov chain, $X$, is \emph{reversible} with respect to the speed measure $\mu^{\om}$, and the jump probabilities of $X$ are given by $p^{\om}(x,y) \ldef \om(\{x,y\}) / \mu^{\om}(x)$, $x, y \in \cC_{\infty}(\om)$.  This walk spends i.i.d.\ $\mathop{\mathrm{Exp}}(1)$-distributed waiting times at all visited vertices and is therefore often called the \emph{constant speed random walk} (CSRW).  We denote by $\Prob_x^{\om}$ the law of the process $X$ starting at the vertex $x \in \cC_{\infty}(\om)$.  For $x, y \in \cC_{\infty}(\om)$ and $t \geq 0$ let $p_{t}^{\om}(x, y)$ be the transition densities of $X$ with respect to the reversible measure (or the \emph{heat kernel} associated with $\cL^{\om}$), i.e.\
\begin{align*}
  p_{t}^{\om}(x, y)
  \;\ldef\;
 \frac{\Prob_x^{\om}\big[X_t = y\big]}{\mu^\om(y)}.
\end{align*}
The heat kernel has been object of very active research in recent years, see \cite{De99, Ba04, BBHK08, BH09, BD10, Fo11, BO12, ADS16a, ADS19} and references therein.

In dimension $d \geq 3$ the behaviour of the Green's function of $X$, defined by
\begin{align} \label{eq:def_Green}
  g^{\om}(x,y) \;\ldef\; \int_0^\infty p_t^{\om}(x, y)\, \md t,
  \qquad x,y \in \cC_{\infty}(\om),
\end{align}
is already quite well understood.  We refer to \cite[Theorem~1.2]{ABDH13} for precise estimates and asymptotics in case of general non-negative i.i.d.\ conductances and to \cite[Theorem~1.14]{ADS16} for a local limit theorem for $g^\om$ in the case of ergodic conductances satisfying a moment condition (cf.\ also \cite[Theorem~5.2]{Ge19}).  Recall that, for every $y \in \cC_{\infty}(\om)$, the function $x \mapsto g^{\om}(x, y)$ is a fundamental solution of $\cL^\om u = -\indicator_{\{y\}}/ \mu^{\om}(y)$.

In the present paper we study the case $d = 2$, which is genuinely different and requires separate consideration.  This is mainly due to the fact that, under suitable conditions, the random walk, $X$, is recurrent in $d = 2$, so the Green kernel $g^{\om}(x, y)$ as in \eqref{eq:def_Green} is ill-defined.  Instead, on $\{0\in \cC_\infty\}$, we  consider the potential kernel
\begin{align} \label{eq:def_a}
  a^{\om}(x,y) 
  \;\ldef\; 
  \lim_{T \to \infty}\,
  \int_0^T \mspace{-6mu}\big( p_t^{\om}(0, 0) - p_t^{\om}(x, y) \big)\, \md t,
  \qquad x,y \in \cC_{\infty}(\om),
\end{align}
whenever the limit exists (cf.\ Theorem~\ref{thm:green1} and Remark~\ref{rem:mainthm}-(i)  below).  Note that $a^{\om}(x,y) = a^{\om}(y,x)$ for all $x,y \in \cC_{\infty}(\om)$, and for every $y \in \cC_{\infty}(\om)$ the function $x \mapsto a^{\om}(x, y)$ is a fundamental solution of $\cL^{\om} u = \indicator_{\{y\}}/ \mu^{\om}(y)$.  The potential kernel plays a crucial role in the potential theory of recurrent Markov processes, cf.~\cite[Chapter~9]{KSK76}, \cite{Sp76,Or64}.
Further, for any $d \geq 2$, the Green's function of the random walk killed upon exiting a finite set $A \subset \bbZ^d$ is given by
\begin{align*}
  g_A^{\om}(x,y) 
  \;\ldef\; 
  \Mean_x^{\om}\!\bigg[ 
    \int_0^{\tau_A} \frac{\indicator_{\{X_t = y\}}}{\mu^{\om}(y)}\, \md t 
  \bigg] 
  \;=\; 
  \int_0^\infty 
    \frac{\Prob_x^{\om}\!\big[X_t = y;\, t < \tau_A \big]}{\mu^{\om}(y)}\, 
  \md t,
\end{align*}
where $\tau_A \ldef \inf\{t > 0 \,:\, X_t \not\in A\}$.  In $d = 2$, $\prob_0$-a.s., for any finite $A \subset \bbZ^2$ we have the following relation between the killed Green kernel and the potential kernel,
\begin{align} \label{eq:rel_pot_green}
  g_A^{\om}(x,y) 
  \;=\; 
  \Mean_x^{\om}\!\big[ a^{\om}(X_{\tau_A}, y) \big] - a^{\om}(x,y),
  \qquad x,y \in A \cap \cC_{\infty}(\om),
\end{align}
and a similar formula holds for infinite sets $A$ (see Lemma~\ref{lem:potgreen} below).  In particular, $\prob_0$-a.s., for the choice $A = (\bbZ^2 \cap \cC_{\infty}(\om) )\setminus \{0\}$ those relations  yield $g_A^{\om}(x,y) = a^{\om}(0,y) - a^{\om}(x,y) + a^{\om}(x,0)$, from which it can be deduced that $\prob_0$-a.s.
\begin{align*}
  \lim_{|y|\to\infty,\, y \in \cC_{\infty}(\om)}  g_A^{\om}(x,y)
  \;=\;
  a^{\om}(0,x), \qquad x\in A,
\end{align*}
see Corollary~\ref{cor:greenZd} below. This provides another interpretation of $a^{\om}(0,x)$ as the expected time  the random walk, starting from infinity,  spends at $x$ before hitting $0$.  Similar statements hold for two-dimensional Brownian motion, cf.\ \cite[Lemmas~3.36 and 8.32]{MP10}. As our first main result we obtain precise asymptotics of the potential kernel and the killed Green kernel under  the  following assumption.
\begin{assumption}\label{ass:static_percolation}
  \begin{enumerate}[(i)]
  \setlength{\itemindent}{0ex}
  \item $\mathbb{P}$ is ergodic, i.e.\ $\mathbb{P}\circ \tau_x^{-1}=\mathbb{P}$ for all $x\in\mathbb{Z}^d$ and $\mathbb{P}(A)\in\{0,1\} \text{ for any } A\in\mathcal{F} \text{ such that }\tau_x(A)=A \text{ for all }x\in\mathbb{Z}^d$.     
  \item Suppose $(\indicator_{\{\om(e)=0\}} : e \in E_d )$ are i.i.d.\ and $\prob[\om(e) > 0] > p_c$.
  \item There exist $p,q \in (1, \infty)$ satisfying $1/p + 1/q < 2/d$ such that
    \begin{align} \label{eq:moment_cond}
      \mean\big[\om(e)^p\big] \;<\; \infty
      \qquad \text{and} \qquad
      \mean\big[\om(e)^{-q} \indicator_{\{e \in \cO\}}\big] \;<\; \infty,
      \qquad \forall\, e \in E_d,
    \end{align}
    where we used the convention that $0/0 = 0$.
  \end{enumerate}
\end{assumption}
Trivially, if $\prob[\om(e)>0] = 1$ then, $\prob$-a.s., $\cC_{\infty} \equiv \bbZ^d$  and therefore $\prob_0 = \prob$.  Under Assumption~\ref{ass:static_percolation} the local limit theorem in \cite[Section~5]{ACS20} gives that, $\prob_0$-a.s.,
\begin{align*}
  \lim_{n \to \infty} n^2 p_{n^2}^{\om}(0, 0) 
  \;=\;
  \frac{1}{2 \pi \sqrt{\det \Si^2}
  \mean\big[ \mu^\om(0) \indicator_{\{0 \in \cC_{\infty}(\om)\}}\big]} 
  \;\rdef\;
  \frac{\bar{g}}{2},
\end{align*}
where $\Si^2$ denotes the covariance matrix of the Brownian motion appearing as the limit process in the quenched invariance principle for $X$ (see \cite{DNS18}).  There is vast further literature on invariance principles for the RCM beyond uniform ellipticity, an incomplete list includes \cite{BP07, Ma08, BD10, ABDH13, ADS15, ACDS18, BR18, BS20}, see also the surveys \cite{Bi11, Ku14}.  Note that, in general, the matrix $\Si^2$ is not diagonal. 

For $x \in \bbZ^d$ we denote by $B(x,r) \ldef \{y \in \bbZ^d : |x-y| < r \}$ balls in $\bbZ^d$ centered at $x$ with respect to the graph distance. We also write $\partial B(x,r) \ldef \{y \in \bbZ^d : |x-y| = r \}$.  For any $x \in \cC_{\infty}(\om)$ let $\cC_n(x,\om) \subset B(x,n) \cap \cC_{\infty}(\om)$ be the connected component of $B(x,n) \cap \cC_{\infty}(\om)$ that contains $x$.  Further, we choose a function $\la_n\!: \bbR^d \to \cC_{\infty}(\om)$ such that $\la_n(x)$ is a closest point in $\cC_{\infty}(\om)$ to $nx$ in the $| \cdot |$-norm. 
\begin{theorem}\label{thm:green1}
  Let $d = 2$ and suppose that Assumption~\ref{ass:static_percolation} is satisfied.  Then, for $\prob_0$-a.e.\ $\om$, $a^\om$ is well-defined, i.e.\ the limit in \eqref{eq:def_a} exists for all $x,y\in \cC_\infty(\om)$, and for any annulus $K = \{x \in \bbR^2 : |x| \in [k_1, k_2]\}$ with $0 < k_1 < k_2 < \infty$, 
  \begin{align} \label{eq:potkernel_main}
    \lim_{n \to \infty}\, \sup_{x \in K} 
    \Big|
    \frac{1}{\ln n}\, 
       a^{\om}\big(0, \la_n(x) \big) \,
    \,-\,
    \bar{g}\mspace{2mu}
    \Big| 
    \;=\; 
    0.
  \end{align}
\end{theorem}
Using relation \eqref{eq:rel_pot_green} we can deduce from Theorem~\ref{thm:green1} the following asymptotics for the killed Green kernel.
\begin{theorem} \label{thm:green_killed}
  Let $d = 2$ and suppose that  Assumption~\ref{ass:static_percolation} is satisfied.
  \begin{enumerate}
  \setlength{\itemindent}{0ex}
  \item[(i)] $\prob_0$-a.s., for any $z \in \cC_{\infty}(\om)$, $\de \in (0,1)$ and $x \in \cC_{(1-\de)n}(z,\om)$,
    \begin{align} \label{eq:green_killed_ondiag}
      \lim_{n \to \infty}\, \frac{1}{\ln n}\, g_{B(z, n)}^{\om}(x, x)
      \;=\;
      \bar{g}.
    \end{align}
  \item[(ii)] $\prob_0$-a.s., for any $z \in \cC_{\infty}(\om)$, $n \in \bbN$, $\de \in (0,1)$ and all $x, y  \in \cC_{(1-\de) n}(z,\om)$ with $x\not =y$,
    \begin{align} \label{eq:green_killed_offdiag}
      \Big| g_{B(z,n)}^{\om}(x, y) - \bar{g} \, \ln \frac{n}{|x-y|} \Big|
      \;\leq\;
      R^{ \om}_{x,y} (n) \,+\, R^{\om}_{x,y} \big(|x-y|\big),
    \end{align}
    for some function $R^\om_{x,y}(\cdot) = R(\tau_x \om ,\tau_y \om, \cdot): \bbN \to [0,\infty)$ which satisfies $R_{x,y}^\om(n)/\ln n \to 0$ as $n\to \infty$ for $\prob_0$-a.e.\ $\om$. 
  \end{enumerate} 
\end{theorem}
\begin{remark} \label{rem:mainthm}
  (i) \emph{Classical random walks.}  For classical (space homogeneous) random walks on $\bbZ^2$ with a transition kernel $p(x,y) = p(0,y-x)$ such that
  \begin{align*}
    \sum\nolimits_{x \in \bbZ^2} x\, p(0,x) \;=\; 0
    \qquad \text{and} \qquad
    \sum\nolimits_{x \in \bbZ^2} |x|^{2 + \de}\, p(0,x) \;<\; \infty
  \end{align*}
  for some $\de > 0$, the existence of the potential kernel, $a$, has been shown in \cite[Proposition~12.1]{Sp76} along with the following precise asymptotics
  \begin{align*}
    \lim_{|x| \to \infty} \big(a(0,x) - \bar{g} \ln |x|\big) \;=\; C,
  \end{align*}
  for some explicit constant $C$  by means of Fourier analysis, cf.\ \cite[Proposition~12.3]{Sp76}. 

  (ii) \emph{Independence of speed measures.} Let $\th^{\om}\!: \bbZ^d \to [0, \infty)$ be a speed measure with $\th^{\om}(x) > 0$ for all $x \in \cC_{\infty}(\om)$ such that $\th^{\om}$ is stationary, i.e.\ $\th^{\om}(x)=\th^{\tau_x \om}(0)$ for all $x \in \cC_{\infty}(\om)$, and $\mean_0[\th^{\om}(0)]<\infty$.  Consider the random walk with generator
  \begin{align*} 
    \big(\cL_\th^{\om} f)(x)
    \;=\; 
    \frac{1}{\th^\om(x)}\, \sum_{y \sim x} \om(\{x,y\})\, \big(f(y) - f(x)\big).
  \end{align*}
  A frequently arising choice for $\th^{\om}$ is the counting measure, i.e.\ $\th^{\om}(x)=1$ for all $x \in \bbZ^d$, associated with the \emph{variable speed random walk} (VSRW).  The various random walks corresponding to different speed measures are time-changes of each other.  Then the Green kernel and the potential kernel $a^{\om}$ do not depend on the speed measure $\th^{\om}$.  In view of \eqref{eq:rel_pot_green} the same applies to the killed Green kernel $g^{\om}_A$. Thus, the constant $\bar{g}$ does not depend on $\th^{\om}$.  As a consequence, under Assumption~\ref{ass:static_percolation}, Theorems~\ref{thm:green1} and \ref{thm:green_killed} also hold for all such speed measures.  In this sense, Theorems~\ref{thm:green1} and \ref{thm:green_killed} are stable under time-changes whenever $\mean_0[\th^{\om} (0)]<\infty$.  This is in contrast to the dynamic RCM considered in Section~\ref{sec:dyn} below.
    
  (iii) In the proof of Theorem~\ref{thm:green1}, Assumption~\ref{ass:static_percolation} is only needed to ensure that upper Gaussian heat kernel bounds, a local limit theorem and H\"older regularity of the heat kernel hold, see \eqref{eq:GB}--\eqref{eq:hoelderreg} below.  Thus, Theorem~\ref{thm:green1} is valid for random walks on any graph and under any field of conductances for which \eqref{eq:GB}--\eqref{eq:hoelderreg} hold.  While in $d=2$ an invariance principle holds under the weaker moment condition $p=q=1$ in \eqref{eq:moment_cond} (see \cite{Bi11}), for the local limit theorem and the Gaussian upper bound the stronger condition in Assumption~\ref{ass:static_percolation} is necessary in  the general ergodic case (see \cite[Section~5]{ADS16} for an example).  The condition can be relaxed in special cases, e.g.\ for i.i.d.\ conductances bounded from above where $q=1/4$ suffices (see \cite[Section~6]{ADS16} and \cite{BKM15}).

  (iv) \emph{RCM on random graphs.} Although Assumption~\ref{ass:static_percolation} reduces $\cC_{\infty}(\om)$ to be a supercritical i.i.d.\ percolation cluster, the proofs of Theorems~\ref{thm:green1} and \ref{thm:green_killed} below can also be extended to a more general class of random graphs under a slightly modified moment condition, cf.\ \cite[Section~5]{ACS20} and \cite{DNS18, Sa17}.    
    
  (v) The upper bound in Theorem~\ref{thm:green_killed}-(ii) constitutes a near diagonal estimate as it only becomes effective in the regime $|x-y|=n^{\al+o(1)}$ for $\al \in (0,1)$ in which case $\lim_{n} \ln(n/|x-y|)/ \ln n = 1-\al > 0$. In the regime  $|x-y| = n^{o(1)}$, \eqref{eq:green_killed_offdiag} corresponds to the on-diagonal behaviour in \eqref{eq:green_killed_ondiag}.  If $|x-y| \geq \de n$ for some $\de>0$, then \eqref{eq:green_killed_offdiag} only gives  $\lim_{n} g^{\om}_{B(z,n)}(x,y)/\ln n=0$, and we expect a more precise statement to follow from a local limit theorem for the killed Green's function, which in turn can be deduced from an invariance principle and an elliptic Harnack inequality.
 
  (vi) From the elliptic Harnack inequality in \cite{ADS16} one can derive a Liouville principle for sublinear harmonic functions (cf.\ \cite{BCDG18}), which together with Theorem~\ref{thm:green1} allows to characterize the potential kernel $a^\om(0,\cdot)$ as the unique solution of $\cL^{\om} u = \indicator_{\{0\}}/ \mu^{\om}(0)$ with logarithmic growth.
    
  (vii) \emph{Annealed Green kernel estimates.} A careful analysis of the proofs of Theorems~\ref{thm:green1} and \ref{thm:green_killed} shows that if we assume in addition $\mean_0\big[\mu^{\om}(0)^{-1} N_1\big] < \infty$ and $\mean_0[ \ln (N_1\vee N_2)] < \infty$, where $N_1$ and $N_2$ are the random constants in the proof of Theorem~\ref{thm:green1} below, then the convergence in \eqref{eq:potkernel_main} and \eqref{eq:green_killed_ondiag} also holds in $L^1(\prob_0)$ and the function $R^{\om}$ in Theorem~\ref{thm:green_killed} (ii) also satisfies $R^{\om}(n)/\ln n \to 0$  in $L^1(\prob_0)$.
\end{remark}
In the i.i.d.\ case similar asymptotics on the killed Green kernel as in Theorem~\ref{thm:green_killed} have been an important ingredient in the proof that the scaling limit of a two-dimensional RCM under heavy-tailed conductances is the fractional kinetics process, see \cite[Proposition~3.1]{Ce11}.  Moreover, it is expected that Theorem~\ref{thm:green_killed} is of importance in the study of the discrete Gaussian free field (DGFF) subject to Dirichlet boundary conditions on a two-dimensional supercritical percolation cluster\footnote{private discussion between S.A.\ and Nathanael Berestycki}. In \cite{Ab15} it has been shown that the irregular structure of $\cC_{\infty}$ affects the growth of the effective resistances which in turn influences the extremal behaviour of the DGFF.

In Section~\ref{sec:dyn} we state the corresponding asymptotics for the quenched and annealed potential kernel under \emph{time-dynamic} conductances (see Theorems~\ref{thm:green_dyn} and \ref{thm:convergence} below). The latter is relevant in the context of the Ginzburg-Landau model for stochastic interfaces (see \cite{Fu05}). In fact, by the Helffer-Sj\"ostrand representation (cf.\ \cite{Fu05, GOS01, DD05}) the variance of the difference of the interface can be expressed in terms of the annealed potential kernel for a particular choice of random dynamic conductances linked to the potential function of the interface, see Section~\ref{subsec:HS}.  Then, Theorem~\ref{thm:convergence} allows to deduce scaling limits for such variances, see Theorem~\ref{thm:HS}.

The rest of the paper is organised as follows. In Section~2 we prove Theorems~\ref{thm:green1} and \ref{thm:green_killed}.  In Section~\ref{sec:dyn} we discuss the corresponding results for the  dynamic RCM. Throughout the paper we write $c$ to denote a positive constant which may change on each appearance.  Constants denoted $c_i$ will be the same through the paper.

\section{Green kernel asymptotics in the two-dimensional static RCM} 
\label{sec:static}

\subsection{Proof of Theorem~\ref{thm:green1} }
We write $d^{\om}(x, y)$ for the graph metric on $\cC_\infty(\om)$ and  $B^{\om}(0,r) \ldef \{y \in \cC_{\infty}(\om) : d^{\om}(0,y) < r \}$ for balls centred at zero with respect to $d^{\om}$.  We first recall that the sizes of the holes in a percolation cluster can be controlled.
\begin{lemma}\label{lem:contr_lambda}  
  For $\prob_0$-a.e.\ $\om$ and any annulus $K = \{x \in \bbR^2 : |x| \in [k_1, k_2]\}$ with $0 < k_1 < k_2 < \infty$, there exist $c_i > 0$ and $N_0 \equiv N_0(\om, k_1,k_2) < \infty$ such that any for $n \geq N_0$ and $x\in K$,
  \begin{align}
    c_{1} n  \;\leq\; |\la_n(x)| \;\leq\; c_{2} n,
    \label{eq:hole1}
  \end{align}
  \vspace{-4ex}
  \begin{align}
    c_{3} n  \;\leq\; d^{\om}\big(0, \la_n(x)\big) \;\leq\; c_{4} n.
    \label{eq:hole2}
  \end{align} 
\end{lemma}
\begin{proof}
  For every $r \geq 1$, let $h^{\om}(r)$ be the size of the biggest 'hole' in $B(0, r) \cap \cC_{\infty}(\om)$, i.e.\ $h^{\om}(r) \ldef \sup\{ r' > 0 : \exists\, y \in Q(0,r) \text{ s.th. } Q(y,r') \cap \cC_{\infty}(\om) = \emptyset\}$, where $Q(x,y) \ldef \{y \in \bbR^2 : |x-y| \leq r\}$.  Then, by \cite[Lemma~5.4]{BH09} (those results are stated for 'holes' and balls w.r.t.\ the maximum norm rather than the equivalent $|\cdot|$-norm used in the present paper), we conclude that, $\prob_0$-a.s., $\lim_{r \to \infty} h^{\om}(r)/r = 0$.  Hence, for any $\de > 0$ there exists $N = N(\om,  k_2, \de)$ such that $|nx - \la_n(x)| \leq h^{\om}(k_2 n) \leq \de n$ for all $n \geq N$ and $x \in K$.  Thus, for $\de = k_1/2$ the triangle inequality implies \eqref{eq:hole1}.

  Further, by  \cite[Lemma~5.3]{BH09}, which is based on arguments in \cite{Ba04}, for each $r > 0$ and $\prob_0$-a.e.\ $\om$, there exist $c > 0$ and $N(\om, r) < \infty$ such that, for $n \geq N(\om, r)$, $d^{\om}(0, \la_n(x)) \leq c |\la_n(x)|$ for all $x \in Q(0,r)$.  Therefore, since trivially $|\la_n(x)| \leq  d^{\om}(0, \la_n(x))$ by the definition of $d^{\om}$, statement \eqref{eq:hole2} follows now from \eqref{eq:hole1}.
\end{proof}
\begin{proof}[Proof of Theorem~\ref{thm:green1}]
  We first recall that under Assumption~\ref{ass:static_percolation} the following three key ingredients for the proof have been established.
  \begin{enumerate}[(i)]
  \setlength{\itemindent}{0ex}
  \item \emph{Upper Gaussian heat kernel bounds} (see \cite[Theorem~1.6]{ADS16a}).  For $\prob_0$-a.e.\ $\om$, there exist $N_1(\om)$ and constants $c_i$ such that for any  $t \geq N_1(\om)$ and all $y \in \cC_\infty(\om)$,
    \begin{align} \label{eq:GB}
      p_t^{\om}(0, y)
      \;\leq\;
      c_{6}\,  t^{-1}\,
      \begin{cases}
        \exp\!\big( \!- c_{7}\, d^\om(0,y)^2/t\big),
        & \text{if $c_{5} d^\om(0,y) \leq  t$, }
        \\[.5ex]
        \exp\!\big(
          \!- c_{8}\, d^\om(0,y) \, \big(1 \vee \log (d^\om(0,y)/t) \big)
        \big),
        &\text{if $c_{5} d^\om(0,y) \geq t$.}
      \end{cases}
    \end{align} 
  \item \emph{Local limit theorem}.  For $\prob_0$-a.e.\ $\om$,
    \begin{align} \label{eq:llt}
      \lim_{n \to \infty} n^2 p_{n^2}^{\om}(0, 0) 
      \;=\;
      \frac{\bar{g}}{2}.
    \end{align}
  \item \emph{H\"older regularity in space}. For $\prob_0$-a.e.\ $\om$, there exist $N_2(\om)$  and positive constants $c_{9}$ and $\varrho$ such that for $R \geq N_2(\om)$ and $\sqrt{T} \geq R$ the following holds.  Setting $T_0 \ldef T+1$ and $R_0^2 \ldef T_0$ we have for any $x_1, x_2 \in B^\om(0, R)$, 
    \begin{align}\label{eq:hoelderreg}
      \big| p_T^{\om}(0, x_1) - p_T^{\om}(0, x_2)\big|
      \;\leq\;
      c_{9}\, \bigg( \frac{R}{\sqrt{T}} \bigg)^{\!\!\vr}\,
      \max_{(s,y)\in [3 T_0 / 4, T_0] \times B^\om(R_0 / 2)} p_s^{\om}(0,y).
    \end{align}
  \end{enumerate}
  In the case $\prob[\om(e)>0]=1$ statements (ii) and (iii) have been established in \cite[Theorem~1.11 and Proposition~4.8]{ADS16}. For variable speed random walks on a general class of random graphs including i.i.d.\ percolation, we refer to \cite[Sections~2.2 and 5]{ACS20}. The results can be easily transferred to the CSRW (see \cite{AT19} for corresponding results on $\bbZ^d$ for a general class of speed measures). Note that for  i.i.d.\ percolation the parameter $\theta\in (0,1)$ appearing in the moment condition in \cite[Section~5]{ACS20} can be chosen arbitrarily, so that the moment condition there coincides with \eqref{eq:moment_cond}.

  The H\"older-regularity is classically deduced from a parabolic Harnack inequality (cf.\ e.g.\ \cite{De99, BH09, ADS16}) or a weaker oscillation inequality as in \cite{ACS20, AT19}.  In conjunction with the near-diagonal heat kernel estimate included in \eqref{eq:GB}, it ensures the existence of the potential kernel $a^\om$, cf.\ \cite[Proof of Lemma~5.2]{GOS01}. We now turn to  the proof of  \eqref{eq:potkernel_main} which we divide into several steps.

  \smallskip 
  \emph{Step 1.} Let $\al \in (0,1)$ be arbitrary.  For $t < N_1(\om)$, using the symmetry of the heat kernel, we have the trivial bound
  \begin{align*}
    \big| p_t^{\om}(0, 0) - p_t^{\om}(0, \la_n(x) ) \big| 
    \;\leq\;
    p_t^{\om}(0,0) + p_t^{\om}(\la_n(x), 0) 
    \;\leq\; 
    2\, \mu^{\om}(0)^{-1},
  \end{align*}
  and for $t \geq N_1(\om)$ we use the on-diagonal part of the estimate in \eqref{eq:GB} to obtain 
  \begin{align} \label{term1}
    &\int_0^{n^\al}\!%
      \big| p_t^{\om}(0, 0) - p_t^{\om}(0, \la_n(x)) \big| \,
    \md t
    \nonumber\\[.5ex]  
    &\mspace{36mu}\leq\;
    2 \mu^{\om}(0)^{-1} N_1(\om) +   
    \int_{N_1(\om)}^{n^\al}\!%
      \big( p_t^{\om}(0, 0) - p_t^{\om}(0, \la_n(x)) \big)\, 
    \md t 
    \nonumber \\[1ex]
    &\mspace{36mu}\leq\; 
    2 \mu^{\om}(0)^{-1} N_1(\om) + c_{9}\, \al\, \ln n.
  \end{align}

  \emph{Step~2.} Let $n \geq N_0(\om)\vee N_2(\om)$ so that $d^\om(0, \la_n(x)) \leq c_{4} n$ for all $x \in K$  by Lemma~\ref{lem:contr_lambda}, and let $\sqrt{t} >  N\ldef  c_{4} \vee  N_1(\om) \vee N_2(\om)$. Then, we use the H\"older regularity in \eqref{eq:hoelderreg} with the choice $T = n^2 t$, $R = c_{4} n$, $x_1 = 0$, $x_2 = \la_n(x)$ and again by the on-diagonal part of \eqref{eq:GB} to obtain
  \begin{align} \label{eq:appl_hoelderreg}
    n^2\, \big| p_{n^2t}^{\om}(0, 0) -p_{n^2 t}^{\om}(0, \la_n(x)) \big|
    &\;\leq\; 
    \frac{c n^2}{t^{\vr/2}}\,
    \max_{(s,y) \in [\frac{3}{4} T_0 , T_0] \times B^\om(0, R_0 / 2)} 
    p_s^{\om}(0, y) 
    \;\leq\;
    \frac{c}{t^{1+\vr/2}},
  \end{align}
  so that
  \begin{align} \label{term2}
    \int_{N n^2}^{\infty}\!%
      \big| p_t^{\om}(0, 0) - p_t^{\om}(0, \la_n(x)) \big|\,
    \md t
    &\;=\;
    \int_{N}^{\infty}\!%
      n^2\, \big| p_{n^2t}^{\om}(0, 0) - p_{n^2t}^{\om}(0, \la_n(x)) \big|\,
    \md t
    \nonumber\\[1ex]
    &\;\leq\;
    c\, \int_{N}^{\infty} \frac{1}{t^{1+\vr/2}}\, \md t 
    \;<\; 
    \infty.
  \end{align}

  \emph{Step 3.} In this step we will show that $\prob_0$-a.s.
  \begin{align} \label{eq:step2}
    \limsup_{n \to \infty}\; \sup_{x \in K}
    \Bigg|
      \frac{1}{\ln n}\,
      \int_{n^\al}^{N n^2}\!%
        \big( p_t^{\om}(0, 0) - p_t^{\om}(0, \la_n(x)) \big)\,
      \md t
      \,-\,
      \bar{g}
    \Bigg| 
    \;\leq \;
    \frac{\bar g}{2} \, \al.
  \end{align}
  The integral can be decomposed into
  \begin{align} \label{eq:decomp_step2}
    &\int_{n^\al}^{N n^2}\!%
      \big( p_t^{\om}(0, 0) - p_t^{\om}(0, \la_n(x)) \big)\,
    \md t  
    \\[.5ex]  
    &\mspace{36mu}=\;  
    \int_{n^{\al-2}}^{N}\!%
      t^{-1} \Big(n^2t\, p_{n^2t}^{\om}(0, 0) - \frac{\bar g}{2} \Big)\,
    \md t 
    \,+\,   
    \frac{\bar g}{2} \,     
    \int_{n^{\al-2}}^{N}\! t^{-1}\, \md t  
    \,-\, 
    \int_{n^{\al}}^{N n^2} p_t^{\om}(0, \la_n(x))\, \md t. \nonumber
  \end{align}
  By the local limit theorem in \eqref{eq:llt}, for any $\de > 0$ there exists $N_3(\om) = N_3(\om, \de)$ such that $|s\, p_s^{\om}(0,0) - \bar g /2| \leq \de$ for all $s \geq N_3(\om)$.  Hence, for $n$ such that $n^{\al} > N_3(\om)$,
  \begin{align} \label{eq:term1}
    \int_{n^{\al-2}}^{N}\mspace{-6mu}
      t^{-1}\, \bigg| n^2t\, p_{n^2t}^{\om}(0, 0) - \frac{\bar g}{2} \bigg|\, 
    \md t 
    \;\leq\; 
    \de \int_{n^{\al-2}}^{N}\mspace{-6mu} t^{-1}\, \md t  
    \;=\;
    \de \big( \ln N + (2-\al) \ln n \big).
  \end{align}  
  Moreover,
  \begin{align} \label{eq:term2}
    \lim_{n \to \infty} 
    \bigg| 
      \frac{1}{\ln n}\, \frac{\bar{g}}{2}\,
      \int_{n^{\al-2}}^{N}\! t^{-1}\, \md t 
      \,-\, 
      \bar{g}
    \bigg|
    \;=\;
    \frac{\bar g}{2} \, \al.     
  \end{align} 
  Let now $\be \in (\al, 2)$ be arbitrary.  Then, for the last term in \eqref{eq:decomp_step2} we get 
  \begin{align} \label{eq:int_middle}
    \int_{n^{\al}}^{N n^2}\! p_t^{\om}(0, \la_n(x))\, \md t 
    &\;=\;
    \int_{n^{\al}}^{n^\be}\! p_t^{\om}(0, \la_n(x))\, \md t
    \,+\,
    \int_{n^{\be}}^{N n^2}\! p_t^{\om}(0, \la_n(x))\, \md t  \nonumber
    \\[1ex]
    &\;\leq\; 
    \int_{n^{\al}}^{n^\be} p_t^{\om}(0, \la_n(x))\, \md t 
    \,+\, 
    c_{6} \big( \ln N + (2-\be) \ln n \big),
  \end{align}
  where we used again the on-diagonal part of \eqref{eq:GB} in the last step.  By Lemma~\ref{lem:contr_lambda}, $d^{\om}(0,\la_n(x)) \geq c_{3} n$ for any $x \in K$ if $n \geq N_0$ , so for such $n$ and  $t \in (n^{\al}, n^{\be})$,
  \begin{align*} 
    p_t^{\om}\big(0, \la_n(x)\big) 
    \;\leq\; 
    c 
    \begin{cases}
      n^{-1} \me^{-c n^{2-\be}}
      &\text{if } t \in \big[ c_{5} \, d^{\om}(0,\la_n(x)), n^{\be} \big],
      \\
      n^{-\al} \me^{-c n}
      &\text{if } t \in \big[n^{\al}, c_{5} \, d^{\om}(0,\la_n(x)) \big].
    \end{cases} 
  \end{align*}
  In view of \eqref{eq:hole2} these bounds immediately imply 
  \begin{align} \label{eq:off_decay}
    \lim_{n \to \infty} \sup_{x \in K} 
    \frac{1}{\ln n}\,
    \int_{n^{\al}}^{n^\be}\! p_t^{\om}\big(0, \la_n(x)\big)\, \md t
    \;=\; 
    0, 
  \end{align}
  and in combination with \eqref{eq:int_middle} this yields
  \begin{align}  \label{eq:term3}
    \limsup_{n\to\infty} \sup_{x\in K} 
    \frac{1}{\ln n}\, \int_{n^{\al}}^{N n^2}\! p_t^{\om}(0, \la_n(x))\, \md t 
    \;\leq\; 
    c_{6} \, (2 - \be).
  \end{align}
  Now we combine \eqref{eq:decomp_step2} with \eqref{eq:term1}, \eqref{eq:term2} and \eqref{eq:term3} to obtain
  \begin{align*}
    &\limsup_{n \to \infty}\; \sup_{x \in K}
    \Bigg|
      \frac{1}{\ln n}\,
      \int_{n^\al}^{N n^2}\!%
        \big( p_t^{\om}(0, 0) - p_t^{\om}(0, \la_n(x)) \big)\,
      \md t
      \,-\,
      \bar{g}
    \Bigg|
   \\
   &\mspace{36mu}\leq\;
   \de (2-\al) +  \frac{\bar{g}}{2}\, \al + c_{6} (2-\be),
  \end{align*}
  and by taking the limits $\de \downarrow 0$ and $\be \uparrow 2$ we get \eqref{eq:step2}.  
  
  \emph{Step 4.} To conclude, note that the combination of \eqref{term1}, \eqref{term2} and \eqref{eq:step2} yields
  \begin{align*}
    \limsup_{n \to \infty}\; \sup_{x \in K}
    \Bigg|
      \frac{1}{\ln n}\,
      \int_{0}^{\infty}\!%
        \big( p_t^{\om}(0, 0) - p_t^{\om}(0, \la_n(x)) \big)\,
      \md t
      \,-\,
      \bar{g}
    \Bigg| 
    \;\leq \;
    c_{9} \al +  \frac{\bar{g}}{2} \, \al ,
  \end{align*}
  and by taking $\al \downarrow 0$ we get \eqref{eq:potkernel_main}.
\end{proof}

\subsection{Proof of Theorem~\ref{thm:green_killed}}
Let $d=2$. The result  will follow  from Theorem~\ref{thm:green1}  and the following relations between the potential kernel and the Green's function of the random walk killed upon exiting a  set $A$ (cf.\ \cite[Proposition~4.6.2(b) and Proposition~4.6.3]{LL10} for the case of a simple random walk in discrete time). 
\begin{lemma} \label{lem:potgreen}
  \begin{enumerate}[(i)]
  \setlength{\itemindent}{0ex}
  \item  $\prob_0$-a.s., for any finite set $A \subset \bbZ^2$ we have for all $x, y \in \cC_{\infty}(\om)$,
    \begin{align} \label{eq:ident_greenA}
      g^{\om}_A(x,y) 
      \;=\; 
      \Mean_x^{\om}\!\big[ a^{\om}(X_{\tau_A}, y) \big] - a^{\om}(x,y).
    \end{align}
  \item  Suppose that  Assumption~\ref{ass:static_percolation} is satisfied.  Then, $\prob_0$-a.s., for any  set $A \subsetneq \cC_{\infty}(\om)$  and all $x, y \in \cC_{\infty}(\om)$,
    \begin{align} \label{eq:ident_greenA2}
      g^{\om}_A(x,y) 
      \;=\; 
      \Mean_x^{\om}\!\big[a^{\om}(X_{\tau_A}, y) \big]
      - a^{\om}(x,y) + f^\om_A(x),
    \end{align}
    where $f^{\om}_A(x) \ldef \lim_{n \to \infty} \bar{g}\, \Prob_x^{\om}\!\big[\tau_{B(0,n)}<\tau_A\big]\, \ln n$.
  \end{enumerate} 
\end{lemma}
\begin{proof}
  (i) Recall that, for any $y \in \cC_{\infty}(\om)$ fixed, $h(z) = a^{\om}(z, y)$ is a fundamental solution of $\cL^{\om} u = \indicator_{\{y\}} / \mu^{\om}(y)$ on $\cC_{\infty}(\om)$, and, under $\Prob_x^{\om}$, the process $(M_t : t \geq 0)$ defined by
  \begin{align*}
    M_t
    \;\ldef\; 
    h(X_t) - \int_0^t \cL^{\om} h(X_s)\, \md s 
    \;=\; 
    h(X_t) - \int_0^t \frac{\indicator_{\{X_s = y\}}}{\mu^{\om}(y)}\, \md s
  \end{align*}
  is a local martingale.  In particular,
  \begin{align*}
    a^{\om}(x, y) 
    \;=\; 
    \Mean_x^{\om}\!\big[ M_0 \big] 
    &\;=\; 
    \Mean_x^{\om}\!\big[ M_{t \wedge \tau_A} \big] 
    \\[0.5ex]
    &\;=\;  
    \Mean_x^{\om}\!\big[ a^{\om}(X_{t \wedge \tau_A}, y) \big] 
    \,-\, 
    \Mean_x^{\om}\!\bigg[%
      \int_0^{t \wedge \tau_A}
        \frac{\indicator_{\{X_s=y\}}}{\mu^{\om}(y)}\,
      \md s 
    \bigg].
  \end{align*}
  Since $A$ is finite, by the dominated convergence theorem
  \begin{align*}
    \lim_{t \to \infty}
    \Mean_x^{\om}\!\big[ a^{\om}(X_{t \wedge \tau_A},y) \big] 
    \;=\; 
    \Mean_x^{\om}\!\big[ a^{\om}(X_{\tau_A},y) \big]
  \end{align*}
  and by the monotone convergence theorem
  \begin{align*}
    \lim_{t \to \infty} 
    \Mean_x^{\om}\!\bigg[%
    \int_0^{t \wedge \tau_A}
      \frac{\indicator_{\{X_s=y\}}}{\mu^{\om}(y)}\,
    \md s 
    \bigg] 
    \;=\; 
    \Mean_x^{\om}\!\bigg[%
      \int_0^{\tau_A} \frac{\indicator_{\{X_s=y\}}}{\mu^{\om}(y)}\, \md s 
    \bigg] 
    \;=\; 
    g^{\om}_A(x, y),
  \end{align*}
  which finishes the proof of (i). Statement (ii) follows from Theorem~\ref{thm:green1} by the same arguments as \cite[Proposition~4.6.3]{LL10}.
\end{proof}
\begin{corro} \label{cor:greenZd}
  Suppose that  Assumption~\ref{ass:static_percolation} is satisfied.  Set $A \ldef (\bbZ^2 \cap \cC_{\infty}(\om) )\setminus \{0\}$.  Then, $\prob_0$-a.s., for all $x,y \in \cC_{\infty}(\om)$,
  \begin{align} \label{eq:greenZdminus0}
    f_{A}^{\om}(x) \;=\; a^{\om}(x,0),
    \qquad 
    g_{A}^{\om}(x,y)
    \;=\;
    a^{\om}(0,y) \,-\,  a^{\om}(x,y) \,+\, a^{\om}(x,0),
  \end{align}
  in particular, $g_{A}^{\om}(x,x) = 2 a^{\om}(0,x) - a^{\om}(x,x)$.  Moreover, $\prob_0$-a.s., for all $x \in \cC_{\infty}(\om)$,
  \begin{align} \label{eq:limGreenZd}
    \lim_{|y|\to\infty,\, y \in \cC_{\infty}(\om)} g_{A}^{\om}(x,y)
    \;=\;
    a^{\om}(x,0).
  \end{align}
\end{corro}
\begin{remark}
  Corollary~\ref{cor:greenZd} extends the formula $g_{\bbZ^2 \setminus \{0\}}(x,x) = 2 a(0,x)$ being valid in the setting of a space-homogeneous random walk on $(\bbZ^2, E_2)$, where $a(x,x) = a(0,0) = 0$, see \cite[Equation~(4.31)]{LL10}.
\end{remark}

\begin{proof}
  In view of \eqref{eq:ident_greenA2} we have
  \begin{align*} 
    g^{\om}_{A}(x,y) 
    \;=\;
    a^{\om}(0,y) - a^{\om}(x,y) + f^{\om}_{A}(x).
  \end{align*}
  In particular, for $y=0$, noting that $g_{A}^{\om}(x,0) = 0$ by its definition, we get $f_{A}^{\om}(x) = a^{\om}(x,0)$.  Hence, \eqref{eq:greenZdminus0} follows directly from \eqref{eq:ident_greenA2}.  By applying \eqref{eq:greenZdminus0} with $x=y$ and using the symmetry of $a^{\om}$ we get $g_{A}^{\om}(x,x) = 2 a^\om(0,x) - a^{\om}(x,x)$.  Finally, to see \eqref{eq:limGreenZd} note that $\lim_{y} p^{\om}_t(z,y) = 0$ for any $z \in \cC_\infty(\om)$ and all $t > 0$, so that
  \begin{align*}
    \lim_{y}
    \big( a^{\om}(0,y) - a^{\om}(x,y) \big)
    \;=\;
    \lim_{y}
    \int_0^\infty \big( p^{\om}_t(0,y) - p^{\om}_t(x,y) \big)\, \md t
    \;=\;
    0
  \end{align*}
  as $|y| \to \infty$ with $y \in \cC_{\infty}(\om)$ by an application of the dominated convergence theorem, which can be justified  by using again the H\"older-regularity in conjunction with the near-diagonal heat kernel estimate.
\end{proof}
\begin{proof}[Proof of Theorem~\ref{thm:green_killed}]
  For any $x, y, z \in \cC_{\infty}(\om)$ such that $x+z, y+z \in \cC_{\infty}(\om)$ we have  $ g_{B(0,n)}^{\tau_z \om}(x,y) = g_{B(z,n)}^{\om}(x+z,y+z)$.  Hence, it suffices to consider the case $z=0$, otherwise we may replace $\om$ by $\tau_z \om$. 
  
  (i) We first show  \eqref{eq:green_killed_ondiag} in the case $x=0$, that is,
  \begin{align} \label{eq:greenkilledondiag0}
    \lim_{n \to \infty}\, \frac{1}{\ln n}\, g_{B(0, n)}^{\om}(0, 0)
    \;=\;
    \bar{g}.
  \end{align}
  Note that $a^{\om}(0, 0) = 0$ and  by \eqref{eq:ident_greenA},
  \begin{align*}
    g^{\om}_{B(0,n)}(0, 0) 
    \;=\; 
    \Mean_0^{\om}\!\big[ a^{\om}(0, X_{\tau_{B(0,n)}}) \big]
    \;=\;  
    \Mean_0^{\om}\!\bigg[%
      \int_0^{\infty}  
        \big( p_t^{\om}(0, 0) - p_t^{\om}(0, X_{\tau_{B(0,n)}}) \big)\, 
      \md t 
    \bigg].
  \end{align*}
  Further, notice that, for every $n$, $X_{\tau_{B(0,n)}} = n y_n = \la_n(y_n)$ for some $y_n$ contained in the annulus $K = \{ u \in \bbR^2 : \frac{1}{2} \leq |u| \leq 2 \}$.  Hence,
  \begin{align*}
    \Big| \frac{1}{\ln n}\, g^{\om}_{B(0,n)}(0, 0) \,-\, \bar g \Big|  
    \;\leq\; 
    \sup_{u \in K} 
    \bigg|
      \frac{1}{\ln n}\,
      \int_0^{\infty}\!%
      \big( p_t^{\om}(0, 0) - p_t^{\om}(0, \la_n(u)) \big)\,
      \md t
      \,-\,
     \bar g
    \bigg|,            
  \end{align*}
  and \eqref{eq:greenkilledondiag0} follows from Theorem~\ref{thm:green1}.  Now, for any  $\de \in (0,1)$ and $x \in \cC_{(1-\de)n}(0,\om)$,
  \begin{align*}
    g^{\om}_{B(x, \delta n/2)} (x,x)
    \;\leq\;
    g^{\om}_{B(0, n)}(x,x)
    \;\leq\;
    g^{\om}_{B(x, 2n)} (x,x)
  \end{align*}
  for $n$ sufficiently large.  Thus, \eqref{eq:green_killed_ondiag} can be derived from \eqref{eq:greenkilledondiag0}.
   
  (ii) Again by \eqref{eq:ident_greenA},
  \begin{align*}
    g^{\om}_{B(0, n)}(x,y)
    &\;=\; 
    \Mean_x^{\om}\!\big[ a^{\om}(y, X_{\tau_{B(0,n)}}) \big] - a^{\om}(y, x)
    \\[.5ex]
    &\;=\;
    \sum_{x' \in \partial B(0,n) \cap \cC_\infty(\om)} 
    \Prob_x^{\om}\!\big[ X_{\tau_{B(0,n)}}=x' \big]\, a^\om(y,x')
    \,-\, a^\om(y,x)
    \\[.5ex]
    &\;=\;
    \sum_{x' \in \partial B(0,n) \cap \cC_\infty(\om)} 
    \Prob_x^{\om}\!\big[ X_{\tau_{B(0,n)}} = x' \big]\, a^{\tau_y \om}(0,x'-y)
    \,-\,  a^{\tau_y \om}(0,x-y).
  \end{align*}
  Recall that $y \in B(0,(1-\de)n) \cap \cC_\infty(\om)$.  Note that, for any $x' \in \partial B(0,n) \cap \cC_\infty(\om)$, $x'-y \in \cC_\infty(\tau_y \om)$ and thus $x'-y = \la_n(y_n)$ for some $y_n \in K \ldef \{u \in \bbR^2 : |u| \in [\delta, 2]\}$.  Hence,
  \begin{align*}
    \sup_{x' \in \partial B(0,n) \cap \cC_\infty(\om)}
    a^{\tau_y \om}(0,x'-y)
    \;\leq\;
    \bar{g}\, \ln n \,+\, R_0^{\tau_y \om}(n),
  \end{align*}
  where $R_0^{\om}\!: \bbN \to [0, \infty)$ is defined as
  \begin{align*}
    R_0^{\om}(n)
    \;\ldef\;
    \sup_{u\in K}
    \Big| a^{\om}\big(0, \la_n(u) \big) \,-\, \bar g \, \ln n \Big|.
  \end{align*}
  Note that $R_0^{\om}(n)/\ln n \to 0$ as $n \to \infty$ for $\prob_0$-a.e.\ $\om$ by Theorem~\ref{thm:green1}. 

  Similarly, setting $N_{xy} \ldef |x-y|$, we may write $x-y \in \cC_\infty(\tau_y \om)$ as $x-y = N_{xy} (x-y)/|x-y| = \la_{N_{xy}}(v)$ with $v = (x-y)/|x-y| \in K$.  Thus,
  \begin{align*}
    a^{\tau_y \om}(0,x-y)
    \;=\;
    a^{\tau_y \om}\big(0, \la_{N_{xy}} (v) \big) 
    \;\geq\;
    \bar{g} \ln\!\big( |x-y| \big) - R_0^{\tau_y \om}(|x-y|).
  \end{align*}
 The combination of the above estimates gives
  \begin{align*}
    g^{\om}_{B(0,n)} (x,y)
    \;\leq\;
    \bar{g}\, \ln\! \Big( \frac{n}{|x-y|} \Big)
    \,+\, R_0^{\tau_y \om}(n) \,+\, R_0^{\tau_y \om}\big(|x-y|\big).
  \end{align*}
  Using a symmetry argument we can replace $ R_0^{\tau_y \om}$ by $R_{x,y}^\om\ldef \frac 1 2 (R_0^{\tau_x \om} + R_0^{\tau_y \om})$.  The corresponding lower bound follows by the same arguments.
\end{proof}

\section{Potential kernel asymptotics for the dynamic RCM} \label{sec:dyn}

\subsection{Setting and results}
In this section we consider the \emph{dynamic} random conductance model.  Let now $\Om$ be the set of measurable functions from $\bbR$ to $(0, \infty)^{E_2}$ equipped with a $\si$-algebra $\cF$ and let $\prob$ be a probability measure on $(\Om, \cF)$.  We will refer to $\om_t(e)$ as the \emph{time-dependent conductance} of the edge $e \in E_2$ at time $t \in \bbR$.  A \emph{space-time shift} by $(s, z) \in \bbR \times \bbZ^2$ is the map $\tau\!: \Om \to \Om$,
\begin{align*}
  (\tau_{s,z}\, \om)_t(\{x,y\})
  \;\ldef\;
  \om_{t+s}(\{x+z,y+z\}),
  \qquad  t \in \bbR,\; \{x,y\} \in E_2.
\end{align*}
The set $\{\tau_{t,x} : (t,x) \in \bbR \times \bbZ^2\}$ together with the operation $\tau_{t,x} \circ \tau_{s,y} = \tau_{t+s,x+y}$ defines the \emph{group of space-time shifts}. Throughout this section we assume that $\prob$ is space-time ergodic.  For any fixed realization $\om \in \Om$, consider  a \emph{time-inhomogeneous} Markov chain, $X = (X_t : t \geq 0)$, on $\bbZ^2$ with time-dependent generator acting on bounded functions $f: \bbZ^2 \to \bbR$ as 
\begin{align}
  \big(\cL_t^{\om} f\big)(x)
  \;=\;
  \sum_{y \sim x}\, \om_t(\{x,y\})\, \big(f(y) - f(x)\big).
\end{align}
 Note that, in contrast to \eqref{def:L},  the total jump rate out of any lattice site is not normalised, and the law of the sojourn time of $X$ depends on its time-space position, i.e.\  $X$ is the \emph{variable speed random walk (VSRW)} with the counting measure as a time-independent invariant measure. The results in this section, as many results on the dynamic RCM,  are restricted to this specific speed measure. We denote by $\Prob_{s,x}^{\om}$ the law of the process starting in $x \in \bbZ^2$  at time $s \geq 0$ and by $p^{\om}_{s,t}(x,y) \ldef \Prob_{s,x}^{\om}\big[X_t = y\big]$ for $x, y \in \bbZ^2$ and $t > s \geq 0$ the heat kernel. 
\begin{assumption}\label{ass:environment_dyn}
  \begin{enumerate}[(i)]
  \setlength{\itemindent}{0ex}
  \item $\prob$ is space-time ergodic, i.e.\ $\prob \circ\, \tau_{t,x}^{-1} \!= \prob\,$ for all $x \in \bbZ^2$, $t\in \bbR$, and $\prob[A] \in \{0,1\}\,$ for any $A \in \cF$ such that $\prob[A \triangle \tau_{t,x}(A)] = 0\,$ for all $x \in \bbZ^2$, $t \in \bbR$. 
  \item For every $A\in \mathcal{F}$ the mapping $(\om,t,x)\mapsto \indicator_A(\tau_{t,x}\om)$ is jointly measurable with respect to the $\sigma$-algebra $\mathcal{F}\otimes \mathcal{B}(\bbR)\otimes \mathcal{P}(\bbZ^2)$.    
  \item There exist $p, q \in (1, \infty]$ satisfying $1/(p-1) + 1/((p-1)q) + 1/q < 2/d$
    such that  $\mean\!\big[\om_t(e)^p\big]< \infty$ and $\mean\!\big[\om_t(e)^{-q}\big] < \infty$ for any $e \in E_2$ and $t \in \bbR$.
  \item \emph{Upper Gaussian heat kernel bounds.} 
    For $\prob$-a.e.\ $\om$, there exist $N_4(\om)$ and constants $c_i$ such that for any given $t$ with $t \geq N_4(\om)$ and all $y \in \bbZ^d$,
    \begin{align*} 
      p_{0,t}^{\om}(0, y)
      \;\leq\;
      c_{11}\,  t^{-1}\,
      \begin{cases}
        \exp\!\big( \!- c_{12}\, |y|^2/t \big),
        & \text{if $c_{10} |y| \leq  t$, }
        \\[.5ex]
        \exp\!\big( \!- c_{13}\, |y|\, \big(1 \vee \log (|y|/t) \big) \big),
        & \text{if $c_{10} |y| \geq t$.}
      \end{cases}
    \end{align*} 
  \end{enumerate}
\end{assumption}  
Under Assumption~\ref{ass:environment_dyn}-(i)--(iii) a quenched invariance principle has been shown in  \cite{ACDS18} (cf.\ also \cite{BR18}). H\"older regularity, near diagonal upper bounds and  a local limit theorem have been shown in \cite{ACS20}. The latter implies that $\lim_{n} n^2 p_{n^2}^{\om}(0,0) = (2 \pi \sqrt{\det \Si^2})^{-1}$, $\prob$-a.s., where $\Si^2$ is the covariance matrix of the Brownian motion in the invariance principle. 

The stronger bounds in Assumption~\ref{ass:environment_dyn}-(iv) will only be used to control the heat kernel in an intermediate time regime (cf.\ the proof of \eqref{eq:off_decay} above). Such bounds are known in the uniformly elliptic case, see \cite[Proposition~4.2]{DD05}.  For unbounded conductances, despite some partial result on the heat kernel decay (see \cite{MO16,GGM19}), the derivation of full Gaussian upper bounds is a subtle open challenge.  One reason is that the dynamic model is restricted to the VSRW.  For a constant speed version not even the invariance principle is known as no time change argument is available.  Even in the static RCM the heat kernel bounds for the VSRW  in \cite{ADS19} are not sufficient since in the degenerate case, in contrast to the CSRW, the intrinsic distance of the VSRW is not comparable to the Euclidean distance in general. However, in the special case of conductances bounded from above, both distances are comparable which  leads to the following example.
\begin{example}
  Set $\om^{*}(e) \ldef \sup_{t} \om_t(e)$, $e \in E_2$.  If $\sup_e \om^{*}(e)<\infty$ (i.e. $p=\infty$ in Assumption~\ref{ass:environment_dyn}-(iii)) and $\mean\!\big[\om_t(e)^{-q}\big] < \infty$ for $q>1$, then the heat kernel bounds in  Assumption~\ref{ass:environment_dyn}-(iv) follow from the arguments in \cite{ADS19}. 
\end{example}
\begin{theorem}\label{thm:green_dyn}
  Suppose that Assumption~\ref{ass:environment_dyn}  holds.  Then, the potential kernel
  \begin{align*}
    a^{\om}(x,y) 
    \;\ldef\; 
    \int_0^{\infty}\!%
      \big( p_{0,t}^{\om}(0, 0) - p_{0,t}^{\om}(x, y) \big)\,
    \md t,
    \qquad x,y \in \bbZ^2, 
  \end{align*}
  is well-defined, and  for any  $K = \{x \in \bbR^2 : |x| \in [k_1, k_2]\}$ with $0 < k_1 < k_2 < \infty$,
  \begin{align*}
    \lim_{n \to \infty}\; \sup_{x\in K} 
    \Big|
      \frac{1}{\ln n}\, a^\om (0, \la_n(x)) \big)    
      \,-\,
      \frac{1}{\pi \sqrt{\det \Si^2}}
    \Big| 
    \;=\; 
    0,
    \qquad \text{ $\prob$-a.s}.
  \end{align*}
\end{theorem}
\begin{proof}
  This follows by similar arguments as in the proof of Theorem~\ref{thm:green1} above.  A local limit theorem and H\"older-regularity have been established in \cite{ACS20}.  The required heat kernel decay is stated in Assumption~\ref{ass:environment_dyn}-(iv).
\end{proof}
We shall also state a corresponding annealed result.  For abbreviation we write $\bar{p}_t(x,y) \ldef \mean\big[p_{0,t}^{\om}(x,y)\big]$ for the averaged transition density. 
\begin{theorem}\label{thm:convergence}
  Suppose that  $\prob\!\big[ c^{-1} < \om_t(e) < c\big] = 1$ for some $c \in[1, \infty)$.  Then, the annealed potential kernel
  \begin{align*}
    \bar{a}(x,y)
    \;\ldef\;
    \int_0^{\infty}\! \big(  \bar{p}_t(0, 0) - \bar{p}_t(x, y) \big)\, \md t,
    \qquad x, y \in \bbZ^2,
  \end{align*}
  is well-defined, and for any  $K = \{x \in \bbR^2 : |x| \in [k_1, k_2]\}$ with $0 < k_1 < k_2 < \infty$,
  \begin{align}
    \lim_{n \to \infty}\; \sup_{x\in K} \bigg|
    \frac{1}{\ln n}\,
    \bar{a}\big(0, \la_n(x) \big)
    \,-\, 
    \frac{1}{\pi \sqrt{\det \Si^2}} \bigg| \;=\; 0.
  \end{align}
\end{theorem}
\begin{proof}
  Again this follows along the lines of the proof of Theorem~\ref{thm:green1}.  As mentioned above, Gaussian bounds have been shown in \cite{DD05} and an annealed local limit theorem has been stated in \cite[Theorem~1.6]{An14}, which has been extended to degenerate conductances in \cite[Theorem~1.11]{AT19}.  Further, an annealed gradient estimate on the heat kernel of the form 
  \begin{align*}
    \big| \bar{p}_{t}(0,x) - \bar{p}_t(0, y) \big| 
    \;\leq\; 
    c\, t^{-3/2},
    \qquad \forall t>0, \, \{x, y\} \in E_2,
  \end{align*}
  has been established in \cite[Theorem~1.6]{CN00} or \cite[Theorem~1.1]{DD05}.  Hence, for $x \in K$ we have by the triangle inequality 
  \begin{align*}
    \big| \bar p_{t}(0, 0) - \bar p_{t}(0, \la_n(x)) \big| 
    \;\leq\; 
    c\, n\, t^{-3/2},
  \end{align*}
  so that for any $N$,
  \begin{align} \label{eq:term2b}
    \int_{N n^2}^{\infty}\!%
      \big| \bar p_t(0, 0) - \bar p_t(0, \la_n(x)) \big|\,
    \md t
    &\;=\; 
    c\, n \int_{N n^2}^{\infty} t^{-3/2}\, \md t 
    \;\leq\; 
    c\, N^{-1/2} 
    \;<\; 
    \infty,
  \end{align}
  which may serve as a replacement for \eqref{term2} and the H\"older regularity estimate. 
\end{proof}

\subsection{Application to stochastic interface models} \label{subsec:HS}  
We briefly outline an application of Theorem~\ref{thm:convergence} in the context of the Ginzburg-Landau $\nabla \phi$ interface model, see \cite{Fu05}.  The interface is described by a field of height variables $\{\phi_t(x) : x \in \bbZ^d, t \geq 0\}$, whose stochastic dynamics are governed by the following infinite system of stochastic differential equations involving nearest neighbour interaction:
\begin{align*} 
  \phi_t(x)
  \;=\;
  \phi(x) \,-\, \int_0^t \sum_{y:|x-y|=1} V'(\phi_s(x)-\phi_s(y)) \, \md s
  \,+\, \sqrt{2}\, w_t(x),
  \qquad x \in \bbZ^d.
\end{align*}
Here $\phi$ is the height of the interface at time $t = 0$, $\{w(x) : x \in \bbZ^d\}$ is a collection of independent Brownian motions and the potential $V \in C^2(\bbR, \bbR_+)$ is even and strictly convex, i.e.\ $c_- \leq V'' \leq c_+$ for some $0 < c_ - < c_ + < \infty$.  Then the formal equilibrium measure for the dynamic is given by the Gibbs measure $Z^{-1} \exp(-H(\phi)) \prod_x d\phi(x)$ on $\bbR^{\bbZ^d}$ with  formal Hamiltonian given by $H(\phi) = \frac{1}{2} \sum_{x\sim y} V(\phi(x)-\phi(y))$.  In dimension $d \geq 3$ this can be made rigorous by taking the thermodynamical limit. In any lattice dimension $d\geq 1$ one considers the gradient process $(\nabla_e \phi_t, : e \in E_d, t\geq 0 )$ instead. Then,  for every $u\in \bbR^d$ describing the tilt of the interface, the gradient  process admits a unique shift invariant ergodic $\nabla \phi$-Gibbs measure $m_u$, see \cite{FS97}. 

By the so-called \emph{Helffer-Sj\"{o}strand representation}  (cf.\ \cite{DD05, GOS01, Fu05}) the variances in the $\nabla \phi$ model can be written in terms of the \emph{annealed} potential kernel of a random walk among dynamic random conductances.  More precisely, for any $x \in \bbZ^d$,
\begin{align}\label{eq:HS}
  \mathrm{var}_{m_u}\big[\phi_0(x) - \phi_0(0)\big]
  \;=\;
  2\, \bar{a}_u(0,x),
\end{align}
where $\bar{a}_u$ denotes the annealed potential kernel (with expectations taken w.r.t.\ $m_u$) associated with  the dynamic RCM with conductances given by
\begin{align} \label{eq:HS_conduct}
  \om_t(x,y)
  \;\ldef\;
  V''\big(\phi_t(y)-\phi_t(x)\big),
  \qquad \{x,y\} \in E_d,\quad t \geq 0.
\end{align}
As an immediate consequence from Theorem~\ref{thm:convergence} we get the following scaling limit.
\begin{theorem} \label{thm:HS}
  Let $d = 2$.  Then, for any ergodic Gibbs measure $m_u$ and any annulus $K = \{x \in \bbR^2 : |x| \in [k_1, k_2]\}$ with $0 < k_1 < k_2 < \infty$,
  \begin{align*}
    \lim_{n \to \infty} \sup_{x \in K} \frac{1}{\ln n}
    \mathrm{var}_{m_u} \big[\phi_0\big(\la_n(x)) - \phi_0(0)\big]
    \;=\;
    \frac{2}{\pi \sqrt{\det \Si_u^2}}.
  \end{align*}
  Here $\Si_u^2$ denotes the covariance matrix in the invariance principle for the random walk under the dynamic random conductances defined in \eqref{eq:HS_conduct}.
\end{theorem}
\begin{proof}
  The conductances in \eqref{eq:HS_conduct} are stationary ergodic under any Gibbs measure $\mu$, and they are uniformly elliptic since the potential function $V$ is assumed to be strictly convex. Hence, Theorem~\ref{thm:convergence} applies and implies the result by \eqref{eq:HS}.
\end{proof}
\begin{remark}
  In $d \geq 3$ the results in \cite{ACDS18, ACS20} can be used to show a scaling limit for the space-time covariances for a class of potentials where the uniform upper bound on the conductances in \eqref{eq:HS_conduct} is replaced by a moment condition (see \cite{AT19}). However, relaxing the lower bound is more challenging as it is required in the existence proof of the Gibbs measure which is based on a Brascamp-Lieb inequality.
\end{remark}

\subsubsection*{Acknowledgment}
We thank the anonymous referees for the careful reading and a number of very constructive suggestions to improve an earlier version of the paper. 

\bibliographystyle{abbrv}
\bibliography{literature}

\end{document}